\documentclass[11 pt]{article}
\usepackage{amsmath,amssymb,amsfonts,amsthm,graphicx,cite,epsfig,makecell,geometry, array,mathrsfs,cite}
\usepackage[varg]{pxfonts}
\usepackage{wrapfig}
\usepackage{enumitem}
\usepackage{epstopdf,footnote,multirow,booktabs,bigints,pstool,booktabs,caption,subcaption ,lipsum}
\usepackage[table]{xcolor}
\usepackage[colorlinks = true,
            linkcolor = red,
            citecolor = blue]{hyperref}
\usepackage[symbol]{footmisc}
\usepackage{titlesec, float}
\usepackage[export]{adjustbox}
\theoremstyle{plain}
\newtheorem{thm}{Theorem}[section]
\newtheorem{lm}[thm]{Lemma}

\theoremstyle{definition}
\newtheorem{de}[thm]{Definition}
\newtheorem{co}[thm]{Corollary}

\theoremstyle{remark}
\newtheorem{re}{\sc \textbf{Remark}}

\numberwithin{equation}{section}
\newcolumntype{?}{!{\vrule width 1.4pt}}

\titlelabel{\thetitle.\quad}
\renewenvironment{abstract}
               {\list{}{\rightmargin\leftmargin}%
                \item[\textbf{\hspace{8.6mm}Abstract ---}]\relax}
               {\endlist}
\DeclareUrlCommand{\url}{%
    \def\UrlLeft##1\UrlRight{\underline{##1}}}

\begin{document}
\title{Upper bound of second Hankel determinant of subclass of bi-univalent functions defined by subordination}
\author{Alaa H. El-Qadeem $^{\text{1}}$ and Mohamed A. Mamon $^{\text{2}}$
\and $^{1}${\small  ahhassan@science.zu.edu.eg\ \& \ } $^{2}${\small mohamed.saleem@science.tanta.edu.eg}\\$^{\text{1}}${\small Department of Mathematics, Faculty of Science, Zagazig University, Zagazig 44519, Egypt}\\$^{\text{2}}${\small Department of Mathematics, Faculty of Science, Tanta University, Tanta 31527, Egypt}}
\date{}
\maketitle

\begin{abstract}
In this paper, new class of bi-univalent functions are introduced. Upper bound of the second Hankel determinant $|H_2(2)|$ of subclass of bi-univalant functions class $\Sigma$, which defined by subordination, investigated. Furthermore, some results concluded as a special case of our main results and corrected some previous researchers results which investigated by miscalculation.
\end{abstract}


\section{Introduction, Definitions and Notations}

Let $\mathcal{A}$ denote the class of all analytic functions $f$ defined in the open unit disk $\mathbb{U}=\{z\in\mathbb{C}:\left\vert z\right\vert <1\}$ and normalized by the condition $f(0)= f^{\prime}(0)-1=0$. Thus each $f\in\mathcal{A}$ has a Taylor-Maclaurin series expansion of the form:
\begin{equation} \label{1.1}
f(z)=z+\sum\limits_{n=2}^{\infty}a_{n}z^{n}, \ \  (z \in\mathbb{U}).
\end{equation}
Further, let $\mathcal{S}$ denote the class of all functions $f \in\mathcal{A}$ which are univalent in $\mathbb{U}$ (for details, see \cite{Duren}; see also some of the recent investigations \cite{C1,C4,C5,C6}).

Two of the important and well-investigated subclasses of the analytic and univalent function class $\mathcal{S}$ are the class
$\mathcal{S}^{\ast}(\alpha)$ of starlike functions of order $\alpha$ in $\mathbb{U}$ and the class $\mathcal{K}(\alpha)$ of convex functions of order $\alpha$ in $\mathbb{U}$. By definition, we have

\begin{equation*}
\mathcal{S}^{\ast}(\alpha):=\left\{ f: \ f \in \mathcal{S} \ \ \text{and} \ \ \mbox{Re}\left\{ \frac{zf^{\prime }(z)}{f(z)}\right\} >\alpha,\quad (z\in \mathbb{%
U}; 0\leq \alpha <1) \right\},
\end{equation*}
and
\begin{equation*}
\mathcal{K}(\alpha):=\left\{ f: \ f \in \mathcal{S} \ \ \text{and} \ \ \mbox{Re}\left\{ 1+\frac{zf^{\prime \prime }(z)}{f^{\prime }(z)}\right\} >\alpha,\quad (z\in \mathbb{%
U}; 0\leq \alpha <1) \right\}.
\end{equation*}
It is clear that $\mathcal{K}(\alpha) \subset \mathcal{S}^{\ast}(\alpha)$. Also we have
\begin{equation*}
f(z) \in \mathcal{K}(\alpha) \ \ \text{iff} \ \ zf^{\prime}(z) \in \mathcal{S}^{\ast}(\alpha),
\end{equation*}
and
\begin{equation*}
f(z) \in \mathcal{S}^{\ast}(\alpha) \ \ \text{iff} \ \ \int_{0}^{z} \frac{f(t)}{t} dt =F(z) \in \mathcal{K}(\alpha).
\end{equation*}

Given two functions $f$ and $h$ in $\mathcal{A}$. The function $f$ is said to be subordinate to $h$ in $\mathbb{U}$, written as $f(z)$ $\prec$
$h(z)$, if there exists a Schwarz function $\omega(z)$, analytic in $\mathbb{U}$, with
\begin{center}
$\omega(0) = 0$ and $\left\vert \omega(z)\right\vert <1$ for all $z \in\mathbb{U}$,
\end{center}
such that $f(z)=h\left(\omega(z)\right)$ for all $z \in\mathbb{U}$. Furthermore, if the function $g$ is univalent in $\mathbb{U}$, then we have the following equivalence (see \cite{mill} and \cite{C2}):%
\[
f(z)\prec h(z)\Leftrightarrow f(0)=h(0)\text{ and }f(\mathbb{U})\subset
h(\mathbb{U}).
\]

Let $\varphi$ be an analytic univalent function in $\mathbb{U}$ with positive real part and $\varphi(\mathbb{U})$ be symmetric with
respect to the real axis, starlike with respect to $\varphi(0) = 1$ and $\varphi^{'}(0) > 0$. Ma
and Minda \cite{Minda} gave a unified presentation of various subclasses of starlike
and convex functions by introducing the classes $\mathcal{S^{*}}(\varphi)$ and $\mathcal{K}(\varphi)$ of functions
$f\in \mathcal{S}$ satisfying $(zf^{'}(z)/f(z)) \prec \varphi(z)$ and $1+(zf^{''}(z)/f^{'}(z))\prec \varphi(z)$ respectively,
which includes several well-known classes as special case. For example, when
$\varphi(z)= (1 + Az)/(1 + Bz)$ with a condition $(1\leq B < A \leq 1)$, the class $\mathcal{S^{*}}(\varphi)$ reduces to
the class $\mathcal{S^{*}}[A,B]$ introduced by Janowski \cite{Janow}. For $0 \leq \beta < 1$, the classes
$\mathcal{S^{*}}(\beta)= \mathcal{S^{*}}((1+(1-2\beta)z)/(1-z))$ and $\mathcal{K}(\beta)= \mathcal{K}((1+(1-2\beta)z)/(1-z))$ are
starlike and convex functions of order $\beta$. Further let $\mathcal{S^{*}}:=\mathcal{S^{*}}(0)$ and $\mathcal{K}:=\mathcal{K}(0)$
are the classes of starlike and convex functions respectively. The class of
strongly starlike functions $\mathcal{S^{*}_{\alpha}}:=\mathcal{S^{*}}({\left((1+z)/(1-z)\right)}^\alpha)$ of order $\alpha(0\leq\alpha<1)$.
The Koebe one quarter theorem \cite{Duren} ensures that the image of $\mathbb{U}$ under
every univalent function $f\in\mathcal{S}$ contains a disk of radius $\frac
{1}{4}$ that every function $f\in \mathcal{S}$ has an inverse map $f^{-1}$ satisfies the following conditions:
\begin{center}
$f^{-1}(f(z))=z \ \ \ (z\in \mathbb{U}),$
\end{center}
and
\begin{center}
$f\left(f^{-1}(w)\right)=w$ $\ \ \ \left( |w|<r_{0}(f);r_{0}(f)\geq\frac{1}{4}\right)$.
\end{center}
In fact, the inverse function is given by
\begin{eqnarray}\label{1.2}
\nonumber g(w)&=&f^{-1}(w)\\
\nonumber &=& w+\sum_{n=2}^{\infty} A_n w^n\\
&=& w-a_2 w^2+(2a_2^2-a_3)w^3-(5a_2^3-5a_2a_3+a_4)w^4+... .
\end{eqnarray}

A function $f\in \mathcal{A}$ is said to be bi-univalent in $\mathbb{U}$ if both $f(z)$ and $f^{-1}(z)$ are univalent in $\mathbb{U}$. Let $\Sigma$ denote the class of bi-univalent functions in $\mathbb{U}$ given by (\ref{1.1}). For a brief history and some interesting examples of functions and characterization of the class $\Sigma$, see Srivastava et al. \cite{Sriv2010}, Frasin and Aouf \cite{C28}, and Magesh and Yamini \cite{F3}.

Examples of functions in the class $\Sigma$ are
\begin{equation*}
\frac{z}{1-z}, \text{ \ \ } -log\left(1-z\right) \text{ \ \ and \ \ }  \frac{1}{2} log\left(\frac{1+z}{1-z}\right) .
\end{equation*}
and so on. However, the familiar Koebe function is not a member of $\Sigma$. Other common examples of functions in $\mathcal{S}$ such as
\begin{equation*}
z-\frac{z^2}{2} \text{ \ \ and \ \ } \frac{z}{1-z^2}
\end{equation*}
are also not members of $\Sigma$.

One of the important tools in the theory of univalent functions are the Hankel determinants, which are
used, for example, in showing that a function of bounded characteristic in U , that is, a function that is a ratio of
two bounded analytic functions, with its Laurent series around the origin having integral coefficients, is rational \cite{Cantor}.

In 1976, Noonan and Thomas \cite{Noonan} defined the $q^{th}$ Hankel determinant for integers $n\geq 1$ and $q\geq1$ by
\begin{equation*}
    H_{q}(n)=\left|\begin{array}{cccccc}
                      a_n & a_{n+1} & ... & a_{n+q-1} \\
                      a_{n+1} & a_{n+2} & ... & a_{n+q} \\
                      . & . & . & . \\
                      . & . & . & . \\
                      . & . & . & . \\
                      a_{n+q-1} & a_{n+q} & ... & a_{n+2q-2}
                    \end{array}
     \right|, \text{ \ \ \ } (a_1=1).
\end{equation*}
Note that

\begin{equation*}
    H_{2}(1)=\left|\begin{array}{cc}
                      a_{1} & a_{2}  \\
                      a_{2} & a_{3}
                    \end{array}
     \right| \text{ \ \ \ and \ \ \ }
     H_{2}(2)=\left|\begin{array}{cc}
                      a_{2} & a_{3}  \\
                      a_{3} & a_{4}
                    \end{array}
     \right|,
\end{equation*}
where the Hankel determinants $H_2(1)=a_3-a_2^2$ and $H_2(2)=a_2a_4-a_3^2$ are well known as Fekete-Szeg\H{o}
and second Hankel determinant functionals, respectively. Furthermore, Fekete and Szeg\H{o} \cite{Fekete} introduced the
generalized functional $a_3-\mu a_2^2$ , where $\mu$ is real number.

Throughout this paper, we assume that $\varphi(z)$ is an analytic function with positive real part in $\mathbb{U}$ and $\varphi(\mathbb{U})$ is symmetric with respect to the real axis. it may have the form
\begin{equation}\label{1.3}
\varphi(z)=1+B_1z+B_2z^2+B_3z^3+...
\end{equation}
with $B_1>0$ and $B_2, B_3$ is any real number.
\begin{de}
Let $\lambda\geq 1, \tau\in \mathbb{C}^{\ast}=\mathbb{C}-\left\{0\right\}$, $0\leq\delta\leq1$ and  let $f\in \Sigma$ given by (\ref{1.1}), then $f$ is said to be in the class $\mathcal{H}_{\Sigma}(\tau,\lambda,\delta;\varphi)$ if it satisfy the following condition
\begin{equation}\label{1.4}
1+\frac{1}{\tau}\left((1-\lambda)\frac{f(z)}{z}+\lambda f^{'}(z)+\delta zf^{''}(z)-1\right)\prec\varphi(z)
\end{equation}
and
\begin{equation}\label{1.5}
1+\frac{1}{\tau}\left((1-\lambda)\frac{g(w)}{w}+\lambda g^{'}(w)+\delta wg^{''}(w)-1\right)\prec\varphi(w).
\end{equation}
where $z,w\in \mathbb{U}$ and $g=f^{-1}\in\Sigma$ given by (\ref{1.2}).
\end{de}

\begin{re}
For special choices of the parameters $\lambda, \tau, \delta$ and the function $\varphi(z)$, we can
obtain the following subclasses as a special case of our class $\mathcal{H}_{\Sigma}(\tau,\lambda,\delta;\varphi)$:
\begin{enumerate}
  \item $\mathcal{H}_{\Sigma}\left(\tau,1,\gamma;\varphi\right)=\Sigma(\tau,\gamma,\varphi)$ which introduced by Tudor \cite{Tudor} and Srivastava and Bansel \cite{Deep}.
  \item $\mathcal{H}_{\Sigma}\left(1,\lambda,0;\varphi\right)=\mathcal{R}_{\sigma}(\lambda,\varphi)$ which defined and studied by Kumar et al. \cite{Ravichandran}.
  \item $\mathcal{H}_{\Sigma}\left(1,\lambda,0;\left(\frac{1+z}{1-z}\right)^{\alpha}\right)=\mathcal{B}_{\Sigma}(\alpha,\lambda)$ and $\mathcal{H}_{\Sigma}\left(1-\beta,\lambda,0;\frac{1+z}{1-z}\right)=\mathcal{B}_{\Sigma}(\beta,\lambda)$ which is introduced by Frasin and Aouf \cite{C28}.
  \item $\mathcal{H}_{\Sigma}\left(1,1,\beta;\left(\frac{1+z}{1-z}\right)^{\alpha}\right)=\mathcal{H}_{\Sigma}(\alpha,\beta)$ and $\mathcal{H}_{\Sigma}\left(1-\gamma,1,\beta;\frac{1+z}{1-z}\right)=\mathcal{H}_{\Sigma}(\gamma,\beta)$ which introduced by Frasin \cite{B.Frasin}.
  \item $\mathcal{H}_{\Sigma}\left(1,1,0;\left(\frac{1+z}{1-z}\right)^{\alpha}\right)=\mathcal{H}_{\Sigma}^{\alpha}=\mathcal{N}_{\sigma}^{\alpha}$ and $\mathcal{H}_{\Sigma}\left(1-\beta,1,0;\frac{1+z}{1-z}\right)=\mathcal{H}_{\Sigma}(\beta)=\mathcal{N}_{\sigma}(\beta)$ which introduced by Srivastava et al. \cite{Sriv2010} and recently studied by \c{C}agler et al. \cite{Deniz}.
  \item $\mathcal{H}_{\Sigma}\left(1-\alpha,\lambda,\delta;\frac{1+z}{1-z}\right)=\mathcal{N}_{\Sigma}(\alpha,\lambda,\delta)$ which introduced by Bulut \cite{Bulut}.
  \item $\mathcal{H}_{\Sigma}\left(\tau,1,\gamma;\frac{1+Az}{1+Bz}\right)=\mathcal{R}_{\gamma,\sigma}^{\tau}(A,B)$ which introduced by Tudor \cite{Tudor}.
\end{enumerate}
\end{re}
For proof of the main results we shall need the following lemma which proved by Kanas et al. \cite{Kanas},
\begin{lm}\label{re1}
If $u(z)=\sum_{n=1}^{\infty} c_n z^n$, $z\in \mathbb{U}$, is a Schawrz function. Then,
\begin{eqnarray*}
c_2 &=& (1-c_1^2)x,\\
c_3 &=& (1-c_1^2)(1-|x|^2)\xi-c_1(1-c_1^2)x^2,
\end{eqnarray*}
for some complex values $x,\xi$ with $|x|\leq 1$ and $|\xi|\leq 1$.
\end{lm}
\begin{lm}\label{re2}
Let $u(z)$ be analytic function in the unit disc $\mathbb{U}$ with $u(0)=0$ and $|u(z)|<1$ for all $z\in U$ with the power series expansion
\[
u(z)=\sum_{n=1}^{\infty} c_n z^n \text{ \ \ ,}z\in U
\]
then $|c_n|\leq 1$ for all $n=1,2,3,...$ . Furthermore, $|c_n|=1$ for some $n(n=1,2,3,...)$ if and only if
\[
u(z)=e^{i\theta} z^n \text{, \ \ \ } \theta \in \mathbb{R}.
\]
\end{lm}
In this present work, we determine the upper bound of the second order Hankel determinant $|H_2(2)|=|a_2a_4-a_3^2|$ of subclass of the class of analytic bi-univalant functions $\Sigma$ which defined by subordination. Furthermore, we introduce some results as
 a special case of our main results and correct some previous results.


\section{Main results}

\begin{thm}\label{thm1}
For $\lambda\geq 1$, $0\leq \delta\leq 1$ and $\tau\in \mathbb{C}-\left\{0\right\}$, let $f(z)$ defined by (\ref{1.1}) belonging to the class $\mathcal{H}_{\Sigma}(\tau,\lambda,\delta;\varphi)$, then
\begin{equation}\label{2.1}
|a_2a_4-a_3^2|\leq B_1|\tau|^2(P+Q+R)
\end{equation}
where,
\begin{eqnarray}\label{5}
\nonumber P &=& \left|\frac{B_3}{(1+\lambda+2\delta)(1+3\lambda+12\delta)}-\frac{B_1^3\tau^2}{(1+\lambda+2\delta)^4}\right|+\frac{B_1}{(1+2\lambda+6\delta)^2}\\
\nonumber &&  +\frac{B_1^2|\tau|}{2(1+\lambda+2\delta)^2(1+2\lambda+6\delta)}+\frac{B_1+2|B_2|}{(1+\lambda+2\delta)(1+3\lambda+12\delta)}\text{ \ ,}\\
 Q &=& \frac{B_1+2|B_2|}{(1+\lambda+2\delta)(1+3\lambda+12\delta)}+\frac{2B_1}{(1+2\lambda+6\delta)^2}+\\
\nonumber && \frac{B_1^2|\tau|}{2(1+\lambda+2\delta)^2(1+2\lambda+6\delta)}\text{ \ ,}\\
\nonumber R &=& \frac{B_1}{(1+2\lambda+6\delta)^2}.
\end{eqnarray}
\end{thm}
\begin{proof}
Since $f(z)\in \mathcal{H}_{\Sigma}(\tau,\lambda,\delta;\varphi)$, then we have
\begin{equation}\label{2.2}
1+\frac{1}{\tau}\left((1-\lambda)\frac{f(z)}{z}+\lambda f^{'}(z)+\delta zf^{''}(z)-1\right)=1+\sum_{n=2}^{\infty}\left(\frac{1+(n-1)(\lambda+n\delta)}{\tau}\right) a_n z^{n-1}
\end{equation}
and since the inverse function $g=f^{-1}$ given by (\ref{1.2})also belonging to the same class, then
\begin{equation}\label{2.3}
1+\frac{1}{\tau}\left((1-\lambda)\frac{g(w)}{w}+\lambda g^{'}(w)+\delta wg^{''}(w)-1\right)=1+\sum_{n=2}^{\infty}\left(\frac{1+(n-1)(\lambda+n\delta)}{\tau}\right) A_n w^{n-1}
\end{equation}

Now, Since $f\in \mathcal{H}_{\Sigma}(\tau,\lambda,\delta;\varphi)$ and $g=f^{-1}\in \mathcal{H}_{\Sigma}(\tau,\lambda,\delta;\varphi)$, by the definition of subordination, there exist two Schwarz functions $u(z)=\sum_{k=1}^{\infty}c_k z^k$ and $v(w)=\sum_{k=1}^{\infty}d_k w^k$ such that
\begin{equation}\label{2.4}
1+\frac{1}{\tau}\left((1-\lambda)\frac{f(z)}{z}+\lambda f^{'}(z)+\delta zf^{''}(z)-1\right)=\varphi(u(z)),
\end{equation}
\begin{equation}\label{2.5}
1+\frac{1}{\tau}\left((1-\lambda)\frac{g(w)}{w}+\lambda g^{'}(w)+\delta wg^{''}(w)-1\right)=\varphi(v(w)).
\end{equation}
By using simple calculations, we can obtain
\begin{equation}\label{2.6}
\varphi(u(z))=1+B_1c_1z+(B_1c_2+B_2c_1^2)z^2+(B_1c_3+2B_2c_1c_2+B_3c_1^3)z^3+... ,
\end{equation}
and
\begin{equation}\label{2.7}
\varphi(v(w))=1+B_1d_1w+(B_1d_2+B_2d_1^2)w^2+(B_1d_3+2B_2d_1d_2+B_3d_1^3)w^3+... .
\end{equation}
By compering the coefficients in both sides of equations (\ref{2.4}) and (\ref{2.5}), we conclude
\begin{equation}\label{2.8}
\frac{1+\lambda+2\delta}{\tau}a_2 = B_{1}c_1,
\end{equation}
\begin{equation}\label{2.9}
-\frac{1+\lambda+2\delta}{\tau}a_2=B_{1}d_1,
\end{equation}
\begin{equation}\label{2.10}
\frac{1+2\lambda+6\delta}{\tau}a_3=B_{1}c_2+B_2c_1^2,
\end{equation}
\begin{equation}\label{2.11}
\frac{1+2\lambda+6\delta}{\tau}(2a_2^2-a_3)=B_{1}d_2+B_2d_1^2,
\end{equation}
\begin{equation}\label{2.12}
\frac{1+3\lambda+12\delta}{\tau}a_4=(B_1c_3+2B_2c_1c_2+B_3c_1^3),
\end{equation}
\begin{equation}\label{2.13}
-\frac{1+\lambda+2\delta}{\tau}(5a_2^3-5a_2a_3+a_4)=(B_1d_3+2B_2d_1d_2+B_3d_1^3).
\end{equation}
From equation (\ref{2.8}) and (\ref{2.9}), we conclude
\begin{equation}\label{2.14}
c_1=- d_1,
\end{equation}
and
\begin{equation}\label{2.15}
a_2 = \frac{B_{1}c_1 \tau}{1+\lambda+2\delta}.
\end{equation}
By subtracting equation (\ref{2.11}) from (\ref{2.10}), we have
\begin{equation}\label{2.16}
a_3=\frac{B_1^2c_1^2\tau^2}{(1+\lambda+2\delta)^2}+\frac{B_1\tau(c_2-d_2)}{2(1+2\lambda+6\delta)},
\end{equation}
Also, by subtracting equation (\ref{2.13}) from (\ref{2.12}), we deduce
\begin{equation}\label{2.17}
2a_4+5a_2^3-5a_2a_3=\frac{\tau}{1+3\lambda+12\delta}\left[B_1(c_3-d_3)+2B_2(c_1c_2-d_1d_2)+B_3(c_1^3-d_1^3)\right].
\end{equation}
By substituting from equation (\ref{2.15}) and (\ref{2.16}) into (\ref{2.17}), we obtain
\begin{eqnarray}\label{2.18}
\nonumber a_4 &=& \frac{5B_1^2c_1\tau^2(c_2-d_2)}{4(1+\lambda+2\delta)(1+2\lambda+6\delta)}+\frac{B_1\tau(c_3-d_3)}{2(1+3\lambda+12\delta)}+\frac{B_3c_1^3\tau}{1+3\lambda+12\delta}\\
&& +\frac{B_2c_1\tau(c_2+d_2)}{1+3\lambda+12\delta}.
\end{eqnarray}

Now, we establish the upper bounds of modules of the second order Hankel determinant $|H_2(2)|$. By Using equations (\ref{2.15}, \ref{2.16}, \ref{2.18}), we have
\begin{eqnarray}\label{2.18}
\nonumber |H_2(2)| &=& |a_2a_4-a_3^2| \\
\nonumber &=&\left|\frac{B_1^3c_1^2\tau^2(c_2-d_2)}{4(1+\lambda+2\delta)^2(1+2\lambda+6\delta)}+\frac{B_1^2c_1\tau^2(c_3-d_3)}{2(1+\lambda+2\delta)(1+3\lambda+12\delta)}\right.\\
&& +\frac{B_1B_2c_1^2\tau^2(c_2+d_2)}{(1+\lambda+2\delta)(1+3\lambda+12\delta)}-\frac{B_1^2\tau^2(c_2-d_2)^2}{4(1+2\lambda+6\delta)^2}+\\
&& \left.\left(\frac{B_1B_3\tau^2}{(1+\lambda+2\delta)(1+3\lambda+12\delta)}-\frac{B_1^4\tau^4}{(1+\lambda+2\delta)^4}\right)c_1^4\right|.
\end{eqnarray}
According to Lemma (\ref{re1}) and using equation (\ref{2.14}), we have
\begin{equation}\label{2.19}
    c_2-d_2=(1-c_1^2)(x-y) \text{ \ \ , \ \ } c_2+d_2=(1-c_1^2)(x+y)
\end{equation}
\begin{equation*}
    c_3=(1-c_1^2)(1-|x|^2)\xi -c_1(1-c_1^2)x^2,\text{ \ \ }d_3=(1-d_1^2)(1-|y|^2)\eta -d_1(1-d_1^2)y^2.
\end{equation*}
Therefore,
\begin{equation}\label{2.20}
    c_3-d_3 = (1-c_1^2)\left[(1-|x|^2)\xi-(1-|y|^2)\eta\right]-c_1(1-c_1^2)(x^2+y^2),
\end{equation}
for some complex values $x,y,\xi$ and $\eta$ with $|x|\leq 1, |y|\leq 1, |\xi|\leq 1$ and $|\eta|\leq 1$.\\
By substituting from equations (\ref{2.19}) and (\ref{2.20}) into (\ref{2.18}), we deduce
\begin{eqnarray}\label{2.21}
\nonumber |H_2(2)| &\leq &B_1|\tau|^2\left[\frac{B_1^2|c_1|^2|\tau|(1+|c_1|^2)(|x|+|y|)}{4(1+\lambda+2\delta)^2(1+2\lambda+6\delta)}\right.\\
\nonumber && +\frac{B_1|c_1|(1+|c_1|^2)\left((2-|x|^2-|y|^2)+|c_1|(|x|^2+|y|^2)\right)}{2(1+\lambda+2\delta)(1+3\lambda+12\delta)}\\
\nonumber && +\frac{|B_2||c_1|^2(1+|c_1|^2)(|x|+|y|)}{(1+\lambda+2\delta)(1+3\lambda+12\delta)}+\frac{B_1(1+|c_1|^2)^2(|x|+|y|)^2}{4(1+2\lambda+6\delta)^2}\\
&& +\left.\left|\frac{B_3}{(1+\lambda+2\delta)(1+3\lambda+12\delta)}-\frac{B_1^3\tau^2}{(1+\lambda+2\delta)^4}\right||c_1|^4\right].
\end{eqnarray}
For more simplification, let $c=|c_1|$ which satisfy $c\leq 1$, then we can say $c\in[0,1]$. So, we conclude that
\begin{eqnarray}\label{2.22}
\nonumber |H_2(2)| &\leq & B_1|\tau|^2\left[\left|\frac{B_3}{(1+\lambda+2\delta)(1+3\lambda+12\delta)}-\frac{B_1^3\tau^2}{(1+\lambda+2\delta)^4}\right|c^4\right.\\
\nonumber && +\frac{c^2(1+c^2)(|x|+|y|)}{1+\lambda+2\delta}\left(\frac{B_1^2|\tau|}{4(1+\lambda+2\delta)(1+2\lambda+6\delta)}+\frac{|B_2|}{1+3\lambda+12\delta}\right)\\
\nonumber && +\frac{B_1c(c-1)(1+c^2)(|x|^2+|y|^2)}{2(1+\lambda+2\delta)(1+3\lambda+12\delta)}+\frac{B_1c(1+c^2)}{(1+\lambda+2\delta)(1+3\lambda+12\delta)}\\
&& +\left.\frac{B_1(1+c^2)^2(|x|+|y|)^2}{4(1+2\lambda+6\delta)^2}\right].
\end{eqnarray}

Now, let $\nu=|x|\leq1$ and $\mu=|y|\leq1$, then we get
\begin{equation}\label{2.23}
|H_2(2)| \leq B_1|\tau|^2 F(\nu,\mu), \text{ \ \ } F(\nu,\mu)=T_1+(\nu+\mu)T_2+(\nu^2+\mu^2)T_3+(\nu+\mu)^2T_4
\end{equation}
where
\begin{eqnarray}
\nonumber T_1 &=& \left|\frac{B_3}{(1+\lambda+2\delta)(1+3\lambda+12\delta)}-\frac{B_1^3\tau^2}{(1+\lambda+2\delta)^4}\right|c^4+\frac{B_1c(1+c^2)}{(1+\lambda+2\delta)(1+3\lambda+12\delta)}\geq 0,\\
\nonumber T_2 &=& \frac{c^2(1+c^2)}{1+\lambda+2\delta}\left(\frac{B_1^2|\tau|}{4(1+\lambda+2\delta)(1+2\lambda+6\delta)}+\frac{|B_2|}{1+3\lambda+12\delta}\right)\geq 0,\\
\nonumber T_3 &=& \frac{B_1c(c-1)(1+c^2)}{2(1+\lambda+2\delta)(1+3\lambda+12\delta)}\leq 0,\\
\nonumber T_4 &=& \frac{B_1(1+c^2)^2}{4(1+2\lambda+6\delta)^2}\geq 0.
\end{eqnarray}

Now, We need to determine the maximum of the function $F(\nu,\mu)$ on the closed square $D=[0,1]\times[0,1]$ for
$c\in[0,1]$. For this work, we must investigate the maximum of $F(\nu,\mu)$ according to $c\in(0,1)$, $c = 0$ and
$c = 1$, taking into the account the sign of $F_{\nu\nu}F_{\mu\mu}-F_{\nu\mu}^{2}$.

First, if we put $c = 0$ , then we have
\begin{equation}\label{2.24}
    F(\nu,\mu)=\frac{B_1}{4(1+2\lambda+6\delta)^2}(\nu+\mu)^2
\end{equation}
So, it is easy to see that
\begin{equation}\label{2.25}
    \operatorname{max}\left\{F(\nu,\mu):(\nu,\mu)\in D\right\}=F(1,1)=\frac{B_1}{(1+2\lambda+6\delta)^2}
\end{equation}

Second, if we put $c = 1$ , then we have
\begin{eqnarray}\label{2.26}
\nonumber F(\nu,\mu)&=& \frac{B_1(\nu+\mu)^2}{(1+2\lambda+6\delta)^2}+\left|\frac{B_3}{(1+\lambda+2\delta)(1+3\lambda+12\delta)}-\frac{B_1^3\tau^2}{(1+\lambda+2\delta)^4}\right|\\
\nonumber && +\frac{(\nu+\mu)}{1+\lambda+2\delta}\left(\frac{B_1^2|\tau|}{2(1+\lambda+2\delta)(1+2\lambda+6\delta)}+\frac{2|B_2|}{1+3\lambda+12\delta}\right)\\
&& +\frac{2B_1}{(1+\lambda+2\delta)(1+3\lambda+12\delta)}.
\end{eqnarray}
So, it is easy to see that
\begin{equation}\label{2.27}
    \operatorname{max}\left\{F(\nu,\mu):(\nu,\mu)\in D\right\}=F(1,1)
\end{equation}
where
\begin{eqnarray}\label{2.26}
\nonumber F(1,1) &=& \frac{4B_1}{(1+2\lambda+6\delta)^2}+\left|\frac{B_3}{(1+\lambda+2\delta)(1+3\lambda+12\delta)}-\frac{B_1^3\tau^2}{(1+\lambda+2\delta)^4}\right|\\
\nonumber && +\frac{B_1^2|\tau|}{(1+\lambda+2\delta)^2(1+2\lambda+6\delta)}+\frac{2(B_1+2|B_2|)}{(1+\lambda+2\delta)(1+3\lambda+12\delta)}
\end{eqnarray}

Finally, let us consider $c\in(0,1)$. Since $T_3<0$ and with some calculations we can ensure that $T_3+2T_4\geq 0$, then we conclude that
\begin{equation}\label{2.27}
    F_{\nu\nu}F_{\mu\mu}-(F_{\nu\mu})^{2}=4T_3(T_3+2T_4)<0
\end{equation}
Therefore, the function $F(\nu,\mu)$ cannot have a local maximum in the interior of the square $D$. Thus, we investigate the maximum on the boundary of the square $D$.
For $\nu=0,$ $0\leq \mu\leq 1$ (similarly for $ \mu=0$ and $0\leq \nu \leq1$)
\begin{eqnarray}\label{2.28}
  F(0,\mu) &=& \Phi(\mu) \\
   &=& T_1+\mu T_2+\mu^2 (T_3+T_4)
\end{eqnarray}

To investigate the maximum of $\Phi(\mu)$, we check the behavior of this function as increasing or decreasing as follow
\begin{equation}\label{2.29}
    \Phi^{'}(\mu)=T_2+2\mu(T_3+T_4)
\end{equation}
\begin{description}
  \item[i] For $T_3+T_4\geq 0$, we obtain $\Phi^{'}(\mu)>0$ which indicate that $\Phi(\mu)$ is an increasing function. Then, the maximum of function $\Phi(\mu)$ occur at $\mu=1$, and so
     \begin{equation*}
     \operatorname{max}\left\{\Phi(\mu):\mu\in[0,1]\right\}=\Phi(1)=T_1+T_2+T_3+T_4
     \end{equation*}
  \item[ii] For $T_3+T_4< 0$, we can easily compute the critical point of the function $\Phi(\mu)$ which given by $\frac{T_2}{2\theta}$ with $\theta=-(T_3+T_4)>0$. Then, we have two cases\newline

      \textbf{Case 1}. If $\frac{T_2}{2\theta}>1$, it follows that $\theta<\frac{T_2}{2}\leq T_2$, and so $T_2+T_3+T_4>0$. Therefore,
      \begin{equation*}
        \Phi(0)=T_1\leq T_1+T_2+T_3+T_4=\Phi(1)\leq T_1+T_2
      \end{equation*}

      \textbf{Case 2}. If $\frac{T_2}{2\theta}\leq1$, it follows that $\frac{T_2^2}{4\theta}\leq \frac{T_2}{2}\leq T_2$ and so $\Phi(1)=T_1+T_2+T_3+T_4\leq T_1+T_2$. Therefore,
      \begin{equation*}
        \Phi(0)=T_1\leq T_1+\frac{T_2^2}{4\theta}=\Phi\left(\frac{T_2}{2\theta}\right)\leq T_1+T_2,
      \end{equation*}
\end{description}
which means that the maximum of the function $\Phi(\mu)$ is not exceed $T_1+T_2$. Thus, we observed that the maximum occurs only at $T_3+T_4\geq0$ which means
\begin{equation}\label{2.30}
     \operatorname{max}\left\{\Phi(\mu):\mu\in[0,1]\right\}=\Phi(1)=T_1+T_2+T_3+T_4
\end{equation}
for any fixed $c\in(0,1)$.
For $\nu=1,$ $0\leq \mu\leq 1$ (similarly for $ \mu=1$ and $0\leq \nu \leq1$)
\begin{eqnarray}\label{2.31}
  F(1,\mu) &=& \Psi(\mu) \\
   &=& T_1+T_2+T_3+T_4+\mu(T_2+2T_4)+\mu^2(T_3+T_4)
\end{eqnarray}

To investigate the maximum of $\Psi(\mu)$, we check the behavior of this function as increasing or decreasing as follow
\begin{equation}\label{2.32}
    \Psi^{'}(\mu)=T_2+2T_4+2\mu(T_3+T_4)
\end{equation}
\begin{description}
  \item[i] For $T_3+T_4\geq 0$, we obtain $\Psi^{'}(\mu)>0$ which indicate that $\Psi(\mu)$ is an increasing function. Hence, the maximum of function $\Psi(\mu)$ occur at $\mu=1$, and so
     \begin{equation*}
     \operatorname{max}\left\{\Psi(\mu):\mu\in[0,1]\right\}=\Psi(1)=T_1+2T_2+2T_3+4T_4
     \end{equation*}
  \item[ii] For $T_3+T_4< 0$, we can easily compute the critical point of the function $\Psi(\mu)$ which given by $\frac{T_2+2T_4}{2\theta}$ with $\theta=-(T_3+T_4)>0$. Then, we have two cases\newline

            \textbf{Case 1}. If $\frac{T_2+2T_4}{2\theta}>1$, it follows that $\theta<\frac{T_2+2T_4}{2}\leq T_2+2T_4$, and so $T_2+T_3+3T_4>0$. Therefore,
      \begin{eqnarray*}
        \Psi(0)&=& T_1+T_2+T_3+T_4 \leq T_1+T_2+T_3+T_4=\Psi(1)\\
        \Psi(1) &\leq& T_1+2T_2+T_3+3T_4
      \end{eqnarray*}

      \textbf{Case 2}. If $\frac{T_2+2T_4}{2\theta}\leq1$, it follows that $\frac{(T_2+2T_4)^2}{4\theta}\leq \frac{T_2+2T_4}{2}\leq T_2+2T_4$ and so $\Psi(1)=T_1+2T_2+2T_3+4T_4\leq T_1+2T_2+T_3+3T_4$. Therefore,
      \begin{eqnarray*}
        \Psi(0)&=& T_1+T_2+T_3+T_4\\
        &\leq& T_1+T_2+T_3+T_4+\frac{(T_2+2T_4)^2}{4\theta}\\
        &=&\Psi\left(\frac{T_2+2T_4}{2\theta}\right)\leq T_1+2T_2+T_3+3T_4,
      \end{eqnarray*}
\end{description}
which means that the maximum of the function $\Psi(\mu)$ is not exceed $T_1+2T_2+T_3+3T_4$. Thus, we observed that the maximum occurs only at $T_3+T_4\geq0$ which means
\begin{equation}\label{2.33}
     \operatorname{max}\left\{\Psi(\mu):\mu\in[0,1]\right\}=\Psi(1)=T_1+2T_2+2T_3+4T_4
\end{equation}
for any fixed $c\in(0,1)$.
Since $\Phi(1)\leq \Psi(1)$ for $c\in[0,1]$, then
\begin{equation}\label{2.34}
   \operatorname{max}\left\{F(\nu,\mu):(\nu,\mu)\in D\right\}=F(1,1)=T_1+2T_2+2T_3+4T_4
\end{equation}
Let the map $\Omega:[0,1]\rightarrow\mathbb{R}$ defined by
\begin{eqnarray}\label{2.35}
    \nonumber \Omega(c)&=&B_1|\tau|^2 \operatorname{max} \{F(\nu,\mu): (\nu,\mu)\in D\}=B_1|\tau|^2 F(1,1)\\
    &=& B_1|\tau|^2(T_1+2T_2+2T_3+4T_4)
\end{eqnarray}
By substituting the values of $T_1, T_2, T_3$ and $T_4$ in the equation (\ref{2.35}), we have
\begin{eqnarray}\label{2.36}
\nonumber \Omega(c) &=& B_1|\tau|^2\left\{c^4 \left[\left|\frac{B_3}{(1+\lambda+2\delta)(1+3\lambda+12\delta)}-\frac{B_1^3\tau^2}{(1+\lambda+2\delta)^4}\right|+\frac{B_1}{(1+2\lambda+6\delta)^2}\right.\right.\\
\nonumber && \left. +\frac{B_1^2|\tau|}{2(1+\lambda+2\delta)^2(1+2\lambda+6\delta)}+\frac{B_1+2|B_2|}{(1+\lambda+2\delta)(1+3\lambda+12\delta)}\right]\\
\nonumber && +c^2\left[\frac{B_1+2|B_2|}{(1+\lambda+2\delta)(1+3\lambda+12\delta)}+\frac{2B_1}{(1+2\lambda+6\delta)^2}+\frac{B_1^2|\tau|}{2(1+\lambda+2\delta)^2(1+2\lambda+6\delta)}\right]\\
          && \left.+\frac{B_1}{(1+2\lambda+6\delta)^2}\right\}
\end{eqnarray}
By setting $c^2=t$ in equation (\ref{2.36}), it may have the form
\begin{equation*}
    \Omega(t)= B_1|\tau|^2(Pt^2+Qt+R), \text{ \ \ \ } t\in[0,1] \text{ \ and \ } R,Q,P\geq 0,
\end{equation*}
where $P,R,Q$ are given by equations (\ref{5}). Then,
\begin{eqnarray}\label{3.37}
   \nonumber \operatorname{max}\Omega(t)&=&\operatorname{max}\{B_1|\tau|^2(Pt^2+Qt+R), t\in[0,1]\}\\
    &=& B_1|\tau|^2(P+Q+R)
\end{eqnarray}
Consequently,
\begin{equation*}
    |H_2(2)|=|a_2a_4-a_3^2|\leq B_1|\tau|^2(P+Q+R).
\end{equation*}
The proof is completed.
\end{proof}

\section{Special results}
By specializing the parameters $\tau,\lambda,\delta$ and the function $\varphi(z)$, we can conclude the upper bound of the second Hankel determinant corresponding to each subclass of our class $\mathcal{H}_{\Sigma}(\tau,\lambda,\delta;\varphi)$ as a special case of our main result.
\begin{co}\label{co1}
Let $f\in \mathcal{R}_\sigma(\lambda,\varphi)$, then
\begin{eqnarray*}
    |a_2a_4-a_3^2|&\leq& B_1\left[\left|\frac{B_3}{(1+\lambda)(1+3\lambda)}-\frac{B_1^3}{(1+\lambda)^4}\right|+\frac{4B_1}{(1+2\lambda)^2}\right.\\
    &&\left.  +\frac{B_1^2}{(1+\lambda)^2(1+2\lambda)}+\frac{2B_1+4|B_2|}{(1+\lambda)(1+3\lambda)}\right].
\end{eqnarray*}
\end{co}
\begin{co}\label{co2}
Let $f\in \mathcal{B}_\Sigma(\alpha,\lambda)$, then
\begin{eqnarray*}
    |a_2a_4-a_3^2|&\leq& 2\alpha\left[\left|\frac{4\alpha^3+2\alpha}{3(1+\lambda)(1+3\lambda)}-\frac{8\alpha^3}{(1+\lambda)^4}\right|+\frac{8\alpha}{(1+2\lambda)^2}\right.\\
    &&\left.  +\frac{4\alpha^2}{(1+\lambda)^2(1+2\lambda)}+\frac{4\alpha+8\alpha^2}{(1+\lambda)(1+3\lambda)}\right].
\end{eqnarray*}
\end{co}
\begin{co}\label{co3}
Let $f\in \mathcal{H}_\Sigma(\alpha,\beta)$, then
\begin{eqnarray*}
    |a_2a_4-a_3^2|&\leq& 2\alpha\left[ \left|\frac{2\alpha^3+\alpha}{12(1+\beta)(1+3\beta)}-\frac{\alpha^3}{4(1+\beta)^4}\right|+\frac{8\alpha}{9(1+2\beta)^2}\right.\\
    &&\left.+\frac{\alpha^2}{3(1+\beta)^2(1+2\beta)}+\frac{\alpha+2\alpha^2}{2(1+\beta)(1+3\beta)}\right].
    \end{eqnarray*}
\end{co}
\begin{co}\label{co4}
Let $f\in \mathcal{N}_{\sigma}^{\alpha}$, then
\begin{equation*}
    |a_2a_4-a_3^2|\leq 2\alpha\left[ \left|\frac{4\alpha^3-\alpha}{12}\right|+\frac{25\alpha}{18}+\frac{4\alpha^2}{3}\right].
    \end{equation*}
\end{co}
\begin{co}\label{co5}
Let $f\in \mathcal{N}_\Sigma(\alpha,\lambda,\delta)$, then
\begin{eqnarray*}
    |a_2a_4-a_3^2|&\leq& 2(1-\alpha)^2\left[\frac{8}{(1+2\lambda+6\delta)^2}+\left|\frac{2}{(1+\lambda+2\delta)(1+3\lambda+12\delta)}-\frac{8(1-\alpha)^2}{(1+\lambda+2\delta)^4}\right|\right. \\
&& \left.+\frac{4(1-\alpha)}{(1+\lambda+2\delta)^2(1+2\lambda+6\delta)}+\frac{12}{(1+\lambda+2\delta)(1+3\lambda+12\delta)}\right].
\end{eqnarray*}
\end{co}
\begin{co}\label{co6}
Let $f\in\mathcal{H}_{\Sigma}(\alpha,\delta)$, then
\begin{eqnarray*}
    |a_2a_4-a_3^2|&\leq& 2(1-\alpha)^2\left[\frac{8}{9(1+2\delta)^2}+\left|\frac{1}{4(1+\delta)(1+3\delta)}-\frac{(1-\alpha)^2}{2(1+\delta)^2}\right|\right. \\
&& \left.+\frac{(1-\alpha)}{3(1+\delta)^2(1+2\delta)}+\frac{3}{2(1+\delta)(1+3\delta)}\right].
\end{eqnarray*}
\end{co}
\begin{co}\label{co7}
Let $f\in\mathcal{N}_{\sigma}(\beta)$, then
\begin{equation*}
    |a_2a_4-a_3^2|\leq 2(1-\alpha)^2\left[\frac{49}{18}-\frac{1}{3}\alpha+\left|\frac{1}{4}-\frac{(1-\alpha)^2}{2}\right|\right].
\end{equation*}
\end{co}
\begin{re}
Previous researchers got wrong results by miscalculation. We corrected their mistakes and obtained the correct result.
\begin{enumerate}
  \item Corollary (\ref{co4}) is a correction of the obtained estimates given in \cite[Theorem 1]{Deniz}.
  \item Corollary (\ref{co7}) is a correction of the obtained estimates given in \cite[Theorem 2]{Deniz}.
\end{enumerate}
\end{re}



\end{document}